\documentclass[12pt]{amsart}
\usepackage[margin=1in]{geometry}
\usepackage[english]{babel}
\usepackage[utf8]{inputenc}
\usepackage{subcaption}
\usepackage{amsmath}
\usepackage{amssymb}
\usepackage{amsfonts}
\usepackage{amsthm}
\usepackage{mathrsfs}
\usepackage[all]{xy}
\usepackage[pdftex]{graphicx}
\usepackage{color}
\usepackage{cite}
\usepackage{url}
\usepackage[labelfont=bf,labelsep=period,justification=raggedright]{caption}
\usepackage[english]{babel}
\usepackage[utf8]{inputenc}
\usepackage{hyperref}
\usepackage[colorinlistoftodos]{todonotes}
\usepackage{tkz-fct}
\usepackage{tikz}
\usetikzlibrary{calc}
\usepackage{multicol}
\PassOptionsToPackage{dvipsnames,svgnames}{xcolor}
\usepackage{textcomp}

\DeclareMathOperator{\rad}{rad}

\newcommand{\eps}{\varepsilon}

\newcommand{\CC}{\mathbb{C}}

\newcommand{\RR}{\mathbb{R}}
\newcommand{\ZZ}{\mathbb{Z}}

\def\multiset#1#2{\ensuremath{\left(\kern-.3em\left(\genfrac{}{}{0pt}{}{#1}{#2}\right)\kern-.3em\right)}}

\theoremstyle{plain}
\newtheorem{thm}{Theorem}
\newtheorem{lemma}[thm]{Lemma}
\newtheorem{cor}[thm]{Corollary}
\newtheorem{conj}[thm]{Conjecture}
\newtheorem{prop}[thm]{Proposition}

\theoremstyle{definition}
\newtheorem{defn}[thm]{Definition}

\theoremstyle{remark}

\numberwithin{equation}{section}
\numberwithin{thm}{section}

\begin{document}

\title{Cyclotomic Coincidences}

\author{Carl Pomerance}
\address{Mathematics Department, Dartmouth College, Hanover, NH 03755}
\email{carl.pomerance@dartmouth.edu}

\author{Simon Rubinstein-Salzedo}
\address{Euler Circle, Palo Alto, CA 94306}
\email{simon@eulercircle.com}

\date{\today}
\maketitle

\begin{abstract}
In this paper, we show that if $m$ and $n$ are distinct positive integers and $x$ is a nonzero real number with $\Phi_m(x)=\Phi_n(x)$, then $\frac{1}{2}<|x|<2$ except when $\{m,n\}=\{2,6\}$ and $x=2$. We also observe that 2 appears to be the largest limit point of the set of values of $x$ for which $\Phi_m(x)=\Phi_n(x)$ for some $m\neq n$.
\end{abstract}

\section{Introduction}

For a positive integer $n$, let $\Phi_n$ denote the $n$th cyclotomic polynomial.
In this paper we consider roots of $\Phi_m(x)-\Phi_n(x)$, where $m,n$ are unequal positive integers.
Our principal theorem is the following.

\begin{thm}\label{thm:principal}
If $m\ne n$ are positive integers and $x$ is a nonzero real number with $\Phi_m(x)=\Phi_n(x)$, then
$\frac12<|x|< 2$, except for $\Phi_2(2)=\Phi_6(2)$.
\end{thm}

We show that on the prime $k$-tuples conjecture the upper bound 2 in the theorem is optimal
in that replacing it with $2-\eps$ for any fixed $\eps>0$, there are infinitely many
counterexamples.

A corollary of Theorem \ref{thm:principal} is the cyclotomic ordering conjecture of
Glasby.  He conjectured that if $m,n$ are positive integers, then either $\Phi_m(q)\le \Phi_n(q)$ for
all integers $q\ge2$ or the reverse inequality holds for all $q$.  This would put a total ordering
on the set of cyclotomic polynomials.  This ordering is also the topic of the sequence A206225 in
the On-Line Encyclopedia of Integer Sequences~\cite{oeis}, where it seems to be tacitly assumed such a total
ordering exists.

In addition, Glasby conjectured that in the total ordering of the cyclotomic polynomials,
$\Phi_{2\cdot3^i}$ is adjacent to $\Phi_{3^i}$ for all $i\ge2$.  We prove a generalization
of this, where 3 may be replaced with any odd prime; see Proposition \ref{prop:pairs}.

\section{Background on Cyclotomic Polynomials}

\begin{defn} For a positive integer $n$, the \emph{$n^\text{th}$ cyclotomic polynomial} $\Phi_n(x)$ is defined as
\[\Phi_n(x)=\prod_{\substack{1\le a\le n \\ \gcd(a,n)=1}} (x-\zeta_n^a),\] where $\zeta_n$ is a primitive $n^\text{th}$ root of unity. \end{defn}

  Let $\phi(n)$ denote Euler's
function at the positive integer $n$, let $\mu(n)$ be the M\"obius function at $n$,
and let $\omega(n)$ denote the number of distinct primes that
divide $n$.  Also, let $\rad(n)$ denote the largest squarefree divisor of $n$ and $q(n)=\frac{n}{\rad(n)}$.
Some familiar facts about cyclotomic polynomials are as follows.

\begin{prop} The degree of $\Phi_n(x)$ is $\phi(n)$. Further, $\Phi_n(x)$ is an irreducible polynomial 
in $\ZZ[x]$.\end{prop}

\begin{prop} \label{prop:mobiusphi} We have 
\[x^n-1=\prod_{d|n}\Phi_d(x)~\hbox{ and }~\Phi_n(x)=\prod_{d\mid n} (x^d-1)^{\mu(\frac{n}{d})}.\] 
When $n>1$, the latter equality can be rewritten as \[\Phi_n(x)=\prod_{d\mid n}(1-x^d)^{\mu(\frac{n}{d})}.\]
\end{prop}

\begin{prop}\label{prop:recip}
When $n>1$, $\Phi_n(x)$ is a reciprocal polynomial; i.e., $\Phi_n(x)=x^{\phi(n)}\Phi_n(1/x)$.\end{prop}

\begin{prop} \label{prop:Phinp} If $p$ is prime and $p\mid n$, then $\Phi_{np}(x)=\Phi_n(x^p)$. 
In general, $\Phi_n(x)=\Phi_{\rad(n)}(x^{q(n)})$.\end{prop}
\begin{prop}\label{prop:neg}
If $n$ is an odd positive integer, and $k\ge2$ is an integer, then
\[ \Phi_n(-x)=\Phi_{2n}(x),~~\Phi_{2n}(-x)=\Phi_n(x),~\hbox{ and }~\Phi_{2^kn}(-x)=\Phi_{2^kn}(x).\]
\end{prop}

\begin{prop} \label{prop:2coeffs}  If $\omega(n)\le 2$, then all the coefficients of $\Phi_n(x)$ lie in $\{-1,0,1\}$. \end{prop}

\section{Rational coincidences}

\begin{thm} \label{thm:valueat2} Suppose $m$ and $n$ are distinct positive integers. Then $\Phi_m(r)\neq\Phi_n(r)$ for rational numbers $r\notin\{-1,0,1\}$ unless $r=2$ and $\{m,n\}=\{2,6\}$. \end{thm}

\begin{proof} For integers $a\ge2$ with $r=a$, the result follows from Bang's Theorem~\cite{Bang86}, which says that if $a,n>1$ are integers and $(a,n)\neq (2,6),(2^j-1,2)$ for some integer $j\ge2$, then there is a prime $p$ such that $p\mid(a^n-1)$ but $p\nmid(a^k-1)$ for any $k<n$. Now, suppose $n>m$ with $n\neq 6$, and let $p$ be a prime dividing $a^n-1$ but not $a^k-1$ for any $k<n$. Then by Proposition~\ref{prop:mobiusphi}, $p\mid\Phi_n(a)$ but $p\nmid\Phi_m(a)$. Thus $\Phi_m(a)\neq\Phi_n(a)$. When $a=2$ and $n=6$, we can just check the values of $\Phi_m(2)$: we have $\Phi_m(2)=1,3,7,5,31$ for $m=1,2,3,4,5$, respectively, while $\Phi_6(2)=3$.  Finally, in the
case of $m=1,n=2$, we see that $\Phi_m(x)-\Phi_n(x)$ has no roots at all.  For integers $a\le-2$, the result
follows from Proposition \ref{prop:neg} by considering the separate cases where $m$ is odd, 2 (mod 4), or
divisible by 4, and the same for $n$.

When $r=a/b\notin\ZZ$, where $a,b$ are coprime integers, we use the generalization of
Bang's theorem due to Zsigmondy~\cite{Zsigmondy92}.  This asserts that $a^n-b^n$ has a prime divisor that
does not divide any $a^k-b^k$ for $1\le k\le n-1$ but for the Bang exceptions.  Let
\[
\Phi_n(x,y)=y^{\phi(n)}\Phi_n(x/y),
\]
so that $\Phi_n(x,y)$ is a homogeneous polynomial with integer coefficients, and as in Proposition
\ref{prop:mobiusphi}, we have
\[
\Phi_n(x,y)=\prod_{d\,\mid\,n}(x^d-y^d)^{\mu(n/d)}.
\]
If $\Phi_m(a/b)=\Phi_n(a/b)$ with $\phi(m)\le\phi(n)$, then $\Phi_m(a,b)=b^{\phi(n)-\phi(m)}\Phi_n(a,b)$,
yet the side of this equation corresponding to the larger of $m,n$ has a prime factor that does not divide the other side.  This completes the
proof.
\end{proof}

\section{An inequality}

The following result will be useful.
\begin{lemma} \label{prop:ineq}
Let $x$ be a real number with $x\ge2$ and let $k$ be a positive integer.  Then
$$
|\log(1-x^{-k})|>\sum_{j>k}|\log(1-x^{-j})|.
$$
\end{lemma}
\begin{proof}
The left side of the inequality is
$$
\sum_{i\ge1}\frac1ix^{-ik},
$$
while the right side is
$$
\sum_{i\ge1}\frac1i\sum_{j>k}x^{-ij}=\sum_{i\ge1}\frac1i\cdot\frac{x^{-ik}}{x^i-1}
<\sum_{i\ge1}\frac1ix^{-ik},
$$
using $x\ge2$.  This completes the proof.
\end{proof}

Note that the following result when $x$ is integral is due to Hering \cite[Theorem 3.6]{Her74}.
\begin{thm}
\label{thm:ineq}
Let $x$ be a real number with $x\ge2$ and let $n$ be a positive integer.  Then
if $\mu(\rad(n))=1$, we have 
\[
\frac{x^{q(n)}-1}{x^{q(n)}}x^{\phi(n)}\le\Phi_n(x)<x^{\phi(n)}
\]
with
equality only in the case $n=1$, while if
$\mu(\rad(n))=-1$, we have 
\[x^{\phi(n)}<\Phi_n(x)<\frac{x^{q(n)}}{x^{q(n)}-1}x^{\phi(n)}.\]
\end{thm}
\begin{proof}
Let $f_n(x)=\Phi_n(x)/x^{\phi(n)}$.
When $n>1$, Propositions \ref{prop:mobiusphi} and \ref{prop:recip} imply that
\begin{equation}\label{eq:fn}
f_n(x)=\Phi_n(x^{-1})=\prod_{d\mid n} (1-x^{-d})^{\mu(\frac{n}{d})}.\end{equation}
This formula continues to hold when $n=1$.

First assume that $n$ is squarefree.  Taking the logarithm of \eqref{eq:fn} we have
\[
\log f_n(x)=\mu(n)\log(1-x^{-1})+\sum_{\substack{d\mid n\\d>1}}\mu(n/d)\log(1-x^{-d}),
\]
so that
\begin{equation}\label{eq:ineq}
\mu(n)\log(1-x^{-1})+\sum_{j>1}\log(1-x^{-j})<\log f_n(x)<\mu(n)\log(1-x^{-1})-\sum_{j>1}\log(1-x^{-j}).
\end{equation}
Thus, by Lemma \ref{prop:ineq}, we have \[(1+\mu(n))\log(1-x^{-1})<\log f_n(x)<(-1+\mu(n))\log(1-x^{-1}).\] Since $\log(1-x^{-1})<0$, we have $f_n(x)>1$ when $\mu(n)=-1$ and $f_n(x)<1$ when
$\mu(n)=1$.  This proves two of the four inequalities of the theorem in the squarefree case.

Still assuming that $n$ is squarefree,  if $p$ is a prime not dividing $n$, then we have \[f_n(x)f_{np}(x)=\prod_{d\mid n} (1-x^{-d})^{\mu(\frac{n}{d})}(1-x^{-d})^{\mu(\frac{pn}{d})}(1-x^{-pd})^{\mu(\frac{pn}{pd})}=\prod_{d\mid n}(1-x^{-pd})^{\mu(\frac{n}{d})}.\] 
We claim first that $f(n)f(np)<1$ if $\mu(n)=1$ and $f(n)f(np)>1$ if $\mu(n)=-1$. To see this, take logarithms, and this is equivalent to saying that 
\[\sum_{d\mid n} \mu\left(\frac{n}{d}\right)\log(1-x^{-pd})<0\] 
if $\mu(n)=1$ and 
\[\sum_{d\mid n}\mu\left(\frac{n}{d}\right)\log(1-x^{-pd})>0\] 
if $\mu(n)=-1$. Let us consider the case where $\mu(n)=1$; the other case is similar. We have 
\begin{align*} \sum_{d\mid n} \mu\left(\frac{n}{d}\right)\log(1-x^{-pd}) &\le \log(1-x^{-p})-\sum_{\substack{d\mid n \\ d>1}} \log(1-x^{-pd})\\
&<\log(1-x^{-p})-\sum_{j>p}\log(1-x^{-j})<0,
\end{align*}
by Lemma \ref{prop:ineq}.

We now complete the proof of the theorem for squarefree numbers by induction on $n$.
 The base case is $n=1$, where we have $f_1(x)=\Phi_1(x)/x=(x-1)/x$, so the theorem
 holds here.  Now, suppose that the result is true for $n$.  We prove it for $np$, where $p$ is a prime
 not dividing $n$. If $\mu(n)=1$, then $\mu(np)=-1$.  To get the upper bound, we have $(x-1)/x\le f_n(x)<1$ and $f_n(x)f_{np}(x)<1$, so \[f_{np}(x)<\frac{1}{f_n(x)}\le
\frac{x}{x-1},\] as desired. The case where $\mu(n)=-1$ is similar.

Finally, we must handle the case where $n$ is not squarefree. 
Using Proposition \ref{prop:Phinp} and noting that
$\phi(n)=q(n)\phi(\rad(n))$, we apply the squarefree case to $\Phi_{\rad(n)}(x^{q(n)})$.
\end{proof}

\begin{cor}
\label{cor:ineq}
Under the same assumptions as in Theorem \ref{thm:ineq}, we have
\[\frac12x^{\phi(n)}\le \Phi_n(x)<x^{\phi(n)}\]
when $\mu(\rad(n))=1$, with equality only in the case $n=1$ and $x=2$.  Else, if $\mu(\rad(n))=-1$,
\[x^{\phi(n)}<\Phi_n(x)<2x^{\phi(n)}.\]
\end{cor}

We now give a proof of a similar result that holds as well for complex numbers.

\begin{prop}\label{prop:complex}
For $z\in\CC$ with $|z|\ge2$, 
\[\frac12|z|^{\phi(n)}\le|\Phi_n(z)|<2|z|^{\phi(n)},\]
with equality only in the cases $n=1$, $z=2$ and $n=2$, $z=-2$.  
\end{prop}

\begin{proof}
Our starting point is \eqref{eq:fn}, which holds as well for complex numbers.  Also, as in the
proof of Theorem \ref{thm:ineq}, it suffices to handle the case when $n$ is squarefree.
The cases $n=1,2$ are true by inspection, so we take $n>2$.
Assume that $\mu(n)=1$; the case when $\mu(n)=-1$ will follow by the same argument.  Let $p$ be the least prime factor of $n$.  By \eqref{eq:fn} we have
\begin{equation}
    \label{eq:ident}
\frac{|\Phi_n(z)|}{|z|^{\phi(n)}}=\frac{|1-z^{-1}|}{|1-z^{-p}|}\prod_{\substack{d|n\\d>p}}|1-z^{-d}|^{\mu(d)},
\end{equation}
using $\mu(n/d)=\mu(d)$ when $n$ is squarefree and $\mu(n)=1$.  
By the triangle inequality, when $|z|\ge2$,
\begin{align}
\frac{|1-z^{-1}|}{|1-z^{-p}|}&=\frac1{|1+z^{-1}+\dots+z^{-(p-1)}|}\ge\frac1{1+|z|^{-1}+\dots+|z|^{-(p-1)}}\notag\\
&\ge\frac1{2-2^{-(p-1)}}
=\frac12(1-2^{-p})^{-1}.
\label{eq:lb}
\end{align}
We now find a lower bound for the remaining product in \eqref{eq:ident}.  For $|z|\ge2$, we have
\[
|(1-z^{-d})^{\mu(d)}|\ge1-2^{-d},
\]
so that
\[
\log\prod_{\substack{d|n\\d>p}}|(1-z^{-d})^{\mu(d)}|\ge\sum_{d>p}\log(1-2^{-d})
>\log(1-2^{-p}),
\]
by Lemma \ref{prop:ineq}.  Hence with \eqref{eq:ident} and \eqref{eq:lb}, the lower bound in the proposition holds.

For the upper bound, first assume that $p>2$.
Referring to \eqref{eq:ident}, note that
\begin{equation}\label{eq:ub}
\frac{|1-z^{-1}|}{|1-z^{-p}|}
=\frac{|z^{p-1}(z-1)|}{|z^p-1|}\le\frac{|z|^{p-1}(|z|+1)}{|z|^p-1}
=1+\frac{|z|^{p-1}+1}{|z|^p-1}\le\frac{12}7,
\end{equation}
when $|z|\ge2$ and $p\ge3$.  
Note that $|(1-z^{-d})^{\mu(d)}|\le(1-|z|^{-d})^{-1}$ for $|z|>1$.  So, for $|z|\ge2$ and referring to 
\eqref{eq:ident},
\[
\prod_{\substack{d>p\\d|n}}|(1-z^{-d})^{\mu(d)}|\le\prod_{d>p}(1-2^{-d})^{-1}
\le\prod_{d\ge4}(1-2^{-d})^{-1}< 1.14.
\]
With \eqref{eq:ub} this completes the upper bound proof when $p\ge3$.

Suppose $p=2$.  Since $n>2$ and $n$ is squarefree,  we may assume that $n$ has an odd prime factor, let
$q$ be the least one.
Again from \eqref{eq:ident} we have
\begin{equation}
    \label{eq:ident2}
\frac{|\Phi_n(z)|}{|z|^{\phi(n)}}=\frac1{|1+z^{-1}||1-z^{-q}|}\prod_{\substack{d|n\\d>q}}|1-z^{-d}|^{\mu(d)}
=\frac{|z|^{q+1}}{|1+z||1-z^{q} |}\prod_{\substack{d|n\\d>q}}|1-z^{-d}|^{\mu(d)}.
\end{equation}
Writing $z=re^{i\theta}$, we have
\begin{align*}
(|1+z|&|1-z^{q}|)^2=(r^2+1+2\Re(z))(r^{2q}+1-2\Re(z^q))\\
&=(r^2+1)(r^{2q}+1)+2r(r^{2q}+1)\cos\theta-2r^q(r^2+1)\cos q\theta-4r^{q+1}\cos\theta\cos q\theta.
\end{align*}
Taking the derivative with respect to $\theta$ and setting it equal to 0 gives us either $\sin\theta=0$
or
\[
2r(r^{2q}+1)=2r^q(r^2+1)q\frac{\sin q\theta}{\sin\theta}+4r^{q+1}\cos q\theta+4r^{q+1}q\cos\theta\frac{\sin q\theta}{\sin\theta}.
\]
If $\sin\theta\ne0$,
using $|\sin q\theta/\sin\theta|< q$ and $r\ge2$, we see that for $q\ge11$ this last equation has no solutions.
So, our expression reaches a minimum at $\theta=0$ or $\theta=\pi$, that is, $z=r$ or $z=-r$.  We see
that $z=-r$ gives the minimum for $|1+z||1-z^q|$.  For $q=3,5,7$ we check directly that the minimum for
$|1+z||1-z^q|$ also occurs at $z=-r$.
Since the logarithmic derivative of $1/((1-r^{-1})(1+r^{-q}))$ as a function of $r$ is negative, this implies that 
\begin{equation}\label{eq:ub2}
\frac{|z|^{q+1}}{|1+z||1-z^{q} |}
\le \frac{r^{q+1}}{(r-1)(r^q+1)}\le
\frac{2^{q+1}}{2^q+1}=2\left(1-\frac1{2^{q}+1}\right).
\end{equation}

Referring to \eqref{eq:ident2}, we thus have for $|z|\ge2$,
\begin{align*}
\log\prod_{\substack{d>q\\d|n}}|(1-z^{-d})^{\mu(d)}|&\le\log\prod_{\substack{d>q\\4\,\nmid\,q}}(1-2^{-d})^{-1}
=\sum_{j\ge1}\frac1j\left(\frac1{2^{qj}(2^j-1)}-\frac1{2^{4j\lfloor q/4\rfloor}(2^{4j}-1)}\right)\\
&<
\frac{13}{15\cdot2^q}+\sum_{j\ge2}\frac1{j2^{qj}(2^j-1)}.
\end{align*}
Since
\[
\log\left(1-\frac1{2^q+1}\right)=-\sum_{j\ge1}\frac1j\frac1{(2^q+1)^j},
\]
with the prior calculation, we see that
\[
\left(1-\frac1{2^q+1}\right)\prod_{\substack{d>q\\d|n}}|(1-z^{-d})^{\mu(d)}|<1.
\]
With \eqref{eq:ident2} and \eqref{eq:ub2}, this completes the proof when $p=2$. \end{proof}

\section{Real coincidences}

In this section we discuss solutions to $\Phi_m(x)=\Phi_n(x)$, where $x\in\RR$, beginning with the
case $x\in(0,1/2]$.

\begin{thm} \label{thm:gmnlemma} Let $m$ and $n$ be distinct positive integers, and let $x$ be a real number with $0<x\le\frac{1}{2}$.  Then $\Phi_m(x)\ne\Phi_n(x)$. \end{thm}

\begin{proof} First, we handle the case where one of $m$ and $n$ is equal to 1, say $m=1$ and $n>1$. Then we have \[\Phi_n(x)=\prod_{d\mid n} (1-x^d)^{\mu(\frac{n}{d})}>0\] whereas $\Phi_1(x)=x-1<0$. Thus $\Phi_n(x)\neq\Phi_1(x)$.

Now assume that $m,n>1$.  Define $g(m,n)=g(m,n,x)$ by 
\[g(m,n)=g(m,n,x)=\log\frac{\Phi_m(x)}{\Phi_n(x)}.\] 
We have 
\begin{equation}\label{eq:gmn}
g(m,n)=\sum_{d\mid m} \mu\left(\frac{m}{d}\right)\log\left(1-x^d\right)-\sum_{e\mid n}\mu\left(\frac{n}{e}\right)\log\left(1-x^e\right).  \end{equation} 
Recall that for a positive integer $k$, we let $q(k)=\frac{k}{\rad(k)}$. We may assume that $q(n)\ge q(m)$.
We split the remainder of the proof up into the following different cases, depending on $m$ and $n$: \begin{itemize} \item $m$ and $n$ are squarefree, \item $m$ is squarefree and $q(n)\ge 4$, \item $m$ is squarefree and $q(n)=3$, \item $m$ is squarefree and $q(n)=2$, \item neither $m$ nor $n$ is squarefree. \end{itemize}

First, assume that $m$ and $n$ are squarefree. Suppose that $\mu(m)\neq\mu(n)$, i.e.\ $\mu(n)=-\mu(m)$. Then the coefficient of $\log(1-x)$ in~\eqref{eq:gmn} is $2\mu(m)$, and the coefficient of each other $\log(1-x^d)$ lies in $\{-2,-1,0,1,2\}$. Hence the sign of $g(m,n)$ is the same as that of $\mu(m)\log(1-x)$ by Lemma~\ref{prop:ineq}, and in particular $g(m,n)\neq 0$. On the other hand, suppose that $\mu(m)=\mu(n)$. Let $d_0$ be the least divisor of either $m$ or $n$
that does not divide $\gcd(m,n)$, and assume without loss of generality that $d_0\mid m$.  
If $d\mid\gcd(m,n)$ then $\mu(m/d)=\mu(n/d)$.  Thus, from \eqref{eq:gmn}
\[
g(m,n)=\mu(m/d_0)\log(1-x^{d_0})+\sum_{\substack{d\,\mid\,m\\d\,\nmid\, n\\d>d_0}}\mu(m/d)\log(1-x^{d})
-\sum_{\substack{e\,|\,n\\e\,\nmid\, m}}\mu(n/e)\log(1-x^{e}).
\]
Since the $d$'s and $e$'s in these two sums are all different, it follows from Lemma \ref{prop:ineq}
that the sign of $g(m,n)$ is the same as the sign of $\mu(m/d_0)\log(1-x^{d_0})$.  In particular,
it is not 0.  Thus, the case when $m,n$ are squarefree is complete.

Next, we tackle the case where $n$ is not squarefree.  In general, \eqref{eq:gmn} reduces to
\begin{equation}
    \label{eq:gmn2}
    g(m,n)=\sum_{d\,|\,\rad(m)}\mu\left(\frac{\rad(m)}d\right)\log(1-x^{dq(m)})
    -\sum_{e\,|\,\rad(n)}\mu\left(\frac{\rad(n)}e\right)\log(1-x^{eq(n)}).
\end{equation}
Assume that $m$ is squarefree (that is, $q(m)=1$) and $n$ is not squarefree
(that is, $q(n)>1$).  As in the proof of Theorem \ref{thm:ineq}, the sum of all of the
$e$-terms is of the same sign as the $e=1$ term and is majorized by that term.  Hence,
if $q(n)\ge4$, we may majorize all of the $e$-terms by a single term with exponent 4
(which doesn't appear in the $d$-sum).  Thus, by Proposition \ref{prop:ineq}, $g(m,n)$
has the same sign as the $d=1$ term, and so is not 0.

Now say $m$ is squarefree and $q(n)=3$.  If $3\nmid m$, then we can majorize the $e$-terms with
a term with exponent 3.  So, assume that $3\mid m$.  We similarly
may assume that $2\mid m$.  
Note that the $d=6$ term appears with the same sign as the $d=1$ term, and the $d=6$ term
majorizes the sum of all higher $d$-terms via Lemma \ref{prop:ineq}.  Assume without essential
loss of generality that $\mu(m)=1$.  Then, majorizing the $e$-terms with an exponent 3 term and
allowing for the possibility of an exponent 5 term, we have
\[
g(m,n)<\log(1-x)-\log(1-x^{2})-2\log(1-x^{3})-\log(1-x^{5}).
\]
Thus,
\begin{align*}
e^{g(m,n)}&<\frac{1-x}{(1-x^{2})(1-x^{3})^2(1-x^{5})}=\frac1{(1+x)(1-x^{3})^2(1-x^{5})}.	
\end{align*}
By inspection, this expression is less than 1 for $0<x\le\frac{1}{2}$.  Thus,
$g(m,n)\ne0$, completing the proof in this case.

Now assume that $m$ is squarefree and $q(n)=2$.  Assume that $\mu(m)=1$; the case $\mu(m)=-1$ is
essentially the same.  There is an $e=2$ term and it appears with the same sign as the
$d=1$ term.
We may assume that $2\mid m$, since otherwise we may replace a putative $d=2$ term with the $e=1$
term and the sum of all $e$-terms with $e>2$ with a putative $d=4$ term.  If $3\nmid m$, we
replace the terms with $e>2$ with a putative $d=3$ term and observe that
\[
g(m,n)<\log(1-x)-2\log(1-x^{2})-\log(1-x^{3})+\log(1-x^{4})-\log(1-x^{5})-\log(1-x^{6}),
\]
so that
\begin{align*}
e^{g(m,n)}&=\frac{1+x^{2}}{(1+x)(1-x^{3})(1-x^{5})(1-x^{6})}\\
&=\frac{1+x^2}{1+x-x^3-x^4-x^5-2x^6-x^7+x^8+2x^9+x^{10}+x^{11}+x^{12}-x^{14}-x^{15}}.
  \end{align*}
Again, by inspection this shows that $e^{g(m,n)}<1$ for $0<x\le\frac{1}{2}$.
  
  Next, assume that $3\mid m$.  Then $d=6$ occurs and with the
same sign as $d=1$.  If $e=3$ occurs, then this too gives a term with exponent 6 and the same
sign as $d=1$, and so the sum of all $e$ terms with $e\ge3$ gives a contribution with the same
sign as $d=1$.  Otherwise, if $3\nmid n$, then the $e$-terms with $e>2$ all have exponent 10 or
greater, and so their sum is majorized
by a term with exponent 9, which is not a $d$-term.  Thus, we have
\begin{multline*}
g(m,n)<\log(1-x)-2\log(1-x^{2})-\log(1-x^{3})+\log(1-x^{4}) \\ -\log(1-x^{5}) +\log(1-x^{6})-\log(1-x^{9}),
\end{multline*} which is smaller than the expression for $g(m,n)$ in the case $3\nmid m$.  Thus, we
have handled the case $q(n)=2$, and so all of the cases with $m$ squarefree.

Now assume that neither $m$ nor $n$ is squarefree and that $1< q(m)\le q(n)$.  First suppose
that $q(m)=q(n)=q$. Then \eqref{eq:gmn2} becomes
\[
g(m,n)=\sum_{d\,|\,\rad(m)}\mu\left(\frac{\rad(m)}d\right)\log(1-x^{dq})
    -\sum_{e\,|\,\rad(n)}\mu\left(\frac{\rad(n)}e\right)\log(1-x^{eq})
\]
and the proof of the case when $m,n$ are both squarefree can be carried over here.

So, assume that $1<q(m)<q(n)$.  As before, assume that $\mu(\rad(m))=1$.
We claim that the $d=1$ term in \eqref{eq:gmn2} dominates all of
the others.  The sum of the $e$-terms is majorized by the $e=1$ term, which has exponent $q(n)\ge3$.
The sum of the $d$ terms
with $d>1$ is majorized by $|\log(1-x^{2q(m)-1})|$.  If $q(m)=2$, we
thus have exponents 2 (from $d=1$), at least 3 (from $d>1$), and $q(n)\ge3$, so that
\[
g(m,n)<\log(1-x^{2})-2\log(1-x^{3}).
\]
Hence,
\[
e^{g(m,n)}<\frac{1-x^{2}}{(1-x^{3})^2}=\frac{1+x}{1+x+x^2-x^3-x^4-x^5}
\]
which is $<1$ for $0<x\le\frac{1}{2}$.
 If $q(m)>2$ the bound is better, so we are done. \end{proof}

\begin{cor} \label{cor:mainthm} Suppose that $x$ is a real number with $x\ge 2$. If $m,n$ are unequal
positive integers, then $\Phi_m(x)\neq\Phi_n(x)$, except when $x=2$ and $\{m,n\}=\{2,6\}$. \end{cor}
\begin{proof} We first note that Corollary~\ref{cor:ineq} immediately gives us the cases when $|\phi(m)-\phi(n)|\ge2$, so we may assume that either $\phi(m)=\phi(n)$ or they are the numbers 1, 2.
In the latter case we quickly verify the sole solution $\Phi_2(2)=\Phi_6(2)$, which leaves $\phi(m)=\phi(n)\ge4$.
If $x\ge2$ and $\Phi_m(x)=\Phi_n(x)$, then Proposition \ref{prop:recip} implies that $\Phi_m(1/x)=\Phi_n(1/x)$,
in violation of Theorem \ref{thm:gmnlemma}.  This completes the proof.
\end{proof}

\begin{cor}\label{cor:neg}
Suppose that $x$ is a real number with either $x\le-2$ or $x\in[-\frac12,0)$.  Then for distinct positive
integers $m,n$, we have $\Phi_m(x)\ne\Phi_n(x)$.
\end{cor}
\begin{proof}
Using Proposition \ref{prop:neg}, this result follows immediately from Theorem \ref{thm:gmnlemma} in
the case that $x\in[-\frac12,0)$ and from Corollary \ref{cor:mainthm} when $x\le-2$.
\end{proof}
Note that Theorem \ref{thm:gmnlemma}, Corollary \ref{cor:mainthm}, and Corollary \ref{cor:neg} immediately give us Theorem \ref{thm:principal}.

\section{An ordering based on cyclotomic polynomials}

A consequence of Corollary~\ref{cor:mainthm} is that we can put an ordering on the positive integers based on the values of cyclotomic polynomials at any $x>2$. More precisely, fix any $x>2$. We write $m\prec n$ if $\Phi_m(x)<\Phi_n(x)$. By Corollary~\ref{cor:mainthm}, $\prec$ is a strict total ordering on the positive integers
which does not depend on the choice of $x$. It is natural to ask about the properties of this ordering.

The first observation is that this ordering is the lexicographic ordering on cyclotomic polynomials. More precisely, suppose $m$ and $n$ are distinct positive integers, and write \[\Phi_m(x)=\sum_{i=0}^\infty a_ix^i,\qquad \Phi_n(x)=\sum_{i=0}^\infty b_ix^i,\] so that $a_i=0$ for $i>\phi(m)$ and $b_i=0$ for $i>\phi(n)$, and each $a_i$ and $b_i$ is an integer. Let $i$ be the smallest integer such that $a_i\neq b_i$. Then $\Phi_m<\Phi_n$ in the lexicographic ordering if $a_i<b_i$, and $\Phi_m>\Phi_n$ if $a_i>b_i$.

\begin{prop} The ordering $\prec$ on the positive integers coincides with the lexicographic ordering on the cyclotomic polynomials. \end{prop}

\begin{proof} Let $f_{m,n}(x)=\Phi_m(x)-\Phi_n(x)$. If $\Phi_m>\Phi_n$ in the lexicographic ordering, then the leading coefficient of $f_{m,n}$ is positive, so for sufficiently large $x$, we have $f_{m,n}(x)>0$. \end{proof}

Note in particular that if $m\prec n$, then $\phi(m)\le\phi(n)$. Thus in the ordering, we first sort the positive integers by their $\phi$-value, and then sort the cyclotomic polynomials lexicographically within each $\phi$-value. Since for any $k$ there are only finitely many positive integers $n$ with $\phi(n)=k$, it follows that the order type of the positive integers with respect to $\prec$ is $\omega$.

It is interesting to identify consecutive pairs in the ordering $\prec$. While this seems to be difficult in general, we can identify certain consecutive pairs.

\begin{prop}\label{prop:pairs}
Let $p$ be an odd prime and $i\ge 2$ an integer. Then $2p^i$ and $p^i$ are consecutive with respect to $\prec$, and $2p^i\prec p^i$.
\end{prop}
We defer the proof until later in this section.

\begin{defn} The \emph{gap} $\gamma(n)$ of $n$ is equal to $\phi(n)-i$, where $i$ is the largest integer less than $\phi(n)$ for which the coefficient of $x^i$ in $\Phi_n(x)$ is nonzero. \end{defn}

\begin{prop} For any positive integer $n$, we have $\gamma(n)=q(n)$. More precisely, for
$x\ge2$, we have \[\Phi_n(x)=x^{\phi(n)}-\mu(\rad(n))x^{\phi(n)-\gamma(n)}+O(x^{\phi(n)-\gamma(n)-1}).\]  \end{prop}

\begin{proof} We first prove that when $n$ is squarefree, then $\gamma(n)=1$ and that \[\Phi_n(x)=x^{\phi(n)}-\mu(\rad(n))x^{\phi(n)-1}+O(x^{\phi(n)-2}).\] We prove this by induction on the number $\omega(n)$ of prime factors of $n$. When $\omega(n)=0$, i.e.\ $n=1$, we have $\Phi_1(x)=x-1$, so the result follows. Next, suppose that the result holds for $n$, where $\omega(n)= k\ge0$, and $p$ is a prime such that $p\nmid n$. Then we have \[\Phi_{np}(x)=\frac{\Phi_n(x^p)}{\Phi_n(x)}.\] By the inductive hypothesis, we thus have 
\begin{align*} \Phi_{np}(x) &= \frac{x^{p\phi(n)}-\mu(n)x^{p\phi(n)-p}+O(x^{p\phi(n)-2p})}{x^{\phi(n)}-\mu(n)x^{\phi(n)-1}+O(x^{\phi(n)-2})} \\ 
&=\frac{x^{(p-1)\phi(n)}+O(x^{(p-1)\phi(n)-p})}{1-\mu(n)x^{-1}+O(x^{-2})} \\ 
&= x^{(p-1)\phi(n)}+\mu(n)x^{(p-1)\phi(n)-1}+O(x^{(p-1)\phi(n)-2}).
\end{align*} 
Since $(p-1)\phi(n)=\phi(pn)$, the proof is complete in the squarefree case.

We reduce the non-squarefree case to the squarefree case using Proposition \ref{prop:Phinp}, completing
the proof.
\end{proof}

We now prove Proposition~\ref{prop:pairs}.

\begin{proof}[Proof of Proposition~\ref{prop:pairs}] It suffices to show that, among all the numbers $n$ with $\phi(n)=\phi(p^i)$, we have $\gamma(p^i)>\gamma(n)$ unless $n\in\{p^i,2p^i\}$. We have $\phi(p^i)=(p-1)p^i$, so if $n$ is not equal to $p^i$ or $2p^i$ but $\phi(n)=\phi(p^i)$, then $n$ must have a prime factor $q$ such that $p\mid (q-1)$. In particular, $q>p$. Now, note that if $n=q_1^{e_1}\cdots q_r^{e_r}$, then we have \[\frac{\gamma(n)}{\phi(n)}=\frac{\prod_{i=1}^rq_i^{e_i-1}}{\prod_{i=1}^r(q_i-1)q_i^{e_i-1}} =\prod_{i=1}^r\frac{1}{q_i-1}<\frac{1}{p-1}=\frac{\gamma(p^i)}{\phi(p^i)}.\] Since $\phi(n)=\phi(p^i)$, we therefore have $\gamma(n)<\gamma(p^i)$, as desired. \end{proof}

If $i=1$, then it is not always true that $2p$ and $p$ are consecutive with respect to $\prec$. However, they are $\prec$-consecutive when $2p$ and $p$ are the only integers whose $\phi$-values are equal to $p-1$. When $p\equiv 3\pmod{4}$, there is a simple criterion.

\begin{prop} Let $p\equiv 3\pmod{4}$ be prime. Then $2p$ and $p$ are $\prec$-consecutive unless there is a prime $q$ and an integer $j\ge 2$ such that $\phi(q^j)=p-1$. \end{prop}

\begin{proof} Note that $p-1\equiv 2\pmod{4}$, and that $\phi(q)$ is even for every odd prime $q$. Since $\phi$ is multiplicative, if $n$ is odd, then $\phi(n)\equiv 0\pmod{2^{\omega(n)}}$. Thus an odd $n$ with $\phi(n)=p-1$ can only have one prime factor. Next, suppose that $n=2^em$, where $m$ is odd and $e\ge 2$. Then $\phi(n)=2^{e-1}\phi(m)$ is divisible by 4 unless $e=2$ and $m=1$, in which case $n$ is a prime power. If $e=1$, then $\phi(n)=\phi(m)$, so we have already analyzed this case. \end{proof}

We remark that very few primes $p\equiv3\pmod 4$ have the property that $p-1=\phi(q^j)$ for some
prime $q$ and exponent $j>1$.  An easy argument shows that the number of such primes $p\le x$ is
$O(\sqrt{x})$. 

When $p\equiv 1\pmod{4}$, there are more ways for there to exist an integer $n$ other than $p$ and $2p$ with $\phi(n)=p-1$. Still, this is relatively unusual behavior: Theorem 4.1 in~\cite{BFLPS05} shows that the number of such primes up to $x$ is $\le\frac{x}{\log^{2+o(1)}x}$ as $x\to\infty$.  On the other hand, it is not known unconditionally
if there are infinitely many such primes, though this follows from Schinzel's Hypothesis H.

There is another total ordering we can put on the positive integers based on the values of cyclotomic polynomials at some $x\in(0,\frac{1}{2}]$, thanks to Theorem~\ref{thm:gmnlemma}. Let us write $m\prec'n$ if $\Phi_m(x)<\Phi_n(x)$ for some (hence any) $x\in(0,\frac{1}{2}]$. Like $\prec$, $\prec'$ is also a lexicographic ordering, but in reverse order of degrees. That is, suppose \[\Phi_m(x)=\sum_{i=0}^\infty a_ix^i,\qquad \Phi_n(x)=\sum_{i=0}^\infty b_ix^i.\] If $m\neq n$, then let $i$ be the smallest nonnegative integer for which $a_i\neq b_i$. Then $m\prec' n$ if $a_i<b_i$, and $n\prec' m$ if $b_i<a_i$.

Unlike $\prec$, $\prec'$ is not a well-ordering. To see this, we produce an infinite decreasing sequence. Let $p$ be any prime. Then for any positive integer $i$, we have $\Phi_{p^i}(x)=1+x^{p^{i-1}}+O(x^{p^{i-1}+1})$ as $x\to 0$. Thus we have $p^{i+1}\prec' p^i$, so the powers of $p$ form an infinite decreasing sequence. 
In addition, the sequence of primes forms an infinite increasing sequence, which implies that the reverse of $\prec'$ is not a well-ordering either.
It would be interesting to describe the order type of $\prec'$.

\section{Near misses}

Other than $\Phi_2(2)=\Phi_6(2)$, we have shown that all real roots of 
of $\Phi_m-\Phi_n$ are smaller than 2. It is natural to ask whether there are roots that get close. To this end, let \[S=\{\alpha\in\RR:\Phi_m(\alpha)=\Phi_n(\alpha)\text{ for some } m\neq n\}.\] Thus we ask whether 2 is a limit point of $S$. We begin with some examples: 
\begin{itemize}
\item $\Phi_{209}-\Phi_{179}$ has a root at $1.99975454398254\cdots$,
\item $\Phi_{407}-\Phi_{359}$ has a root at $1.99975550093366\cdots$, 
\item $\Phi_{221}-\Phi_{191}$ has a root at $1.99993512065828\cdots$,
\item $\Phi_{527}-\Phi_{479}$ has a root at $1.99999618493891\cdots\,$. \end{itemize}

These near-misses were constructed as follows: let $p,q,r$ be primes such that $pq=p+q+r$, and $p<q$. Then we claim that $\Phi_{pq}-\Phi_r$ has a root very close to the largest real root of $\psi_{p-1}(x)=x^{p-1}-x^{p-2}-x^{p-3}\cdots-x-1$, with this root getting closer the larger that $q$ is. Note that the latter polynomial has a root very close to 2, since $\psi_{p-1}(2)=1$ and $\psi'_{p-1}(2)=2^{p-1}-1$, so the largest real root of $\psi_{p-1}$ is approximately $2-\frac{1}{2^{p-1}-1}$. Let us write $\alpha_{p-1}$ for the largest real root of $\psi_{p-1}$.

The reason why $\Phi_{pq}-\Phi_r$ has a root very close to $\alpha_{p-1}$ is that we have a near-factorization of $\Phi_{pq}-\Phi_r$, namely 
\[\Phi_{pq}(x)-\Phi_r(x)=\psi_{p-1}(x)x^{\phi(pq)-\phi(p)}
+\delta(x),\] 
where $\deg(\delta)\le \phi(pq)-p$. Furthermore, by Proposition~\ref{prop:2coeffs}, all the coefficients of $\delta$ lie in $\{-2,-1,0,1\}$. Note that the degree of $\delta$ is much smaller than the degree of the main term $\psi_{p-1}(x)x^{\phi(pq)-\phi(p)}$,
so this is a small perturbation.

In general, suppose we have a polynomial $f(x)$ all of whose complex roots are distinct, and a perturbation polynomial $g(x)$ with $\deg(g)<\deg(f)$. Let us suppose that $f(x)+tg(x)$ factors as $\prod_{i=1}^d (x-\beta_i(t))$, where the $\beta_i$'s are continuous functions for small values of $t$. Then we have \[\beta_i'(0)=-\frac{g(\beta_i(0))}{f'(\beta_i(0))}\] (see~\cite{Wilkinson84}). In our case, with $g=\delta$, we expect to have a root of $\Phi_{pq}-\Phi_r$ near \[\alpha_{p-1}-\frac{\delta(\alpha_{p-1})}{\Phi_{pq}'(\alpha_{p-1})-\Phi_r'(\alpha_{p-1})-\delta'(\alpha_{p-1})}.\] Since $|\delta(\alpha_{p-1})|\le 2^{pq-2p-q+3}$ and the denominator has size on the order of $2^{pq-p-q}$, we have a root of $\Phi_{pq}-\Phi_r$ somewhere around \[\alpha_{p-1}-\frac{1}{2^q}.\] This matches experimental observation, as shown in Table~\ref{tab:erroranalysis}. Here $\beta$ is the root of $\Phi_{pq}-\Phi_r$ close to $\alpha_{p-1}$.

\begin{table}
    \begin{tabular}{ccc|cccc}
    $p$ & $q$ & $r$ & $\beta$ & $\alpha_{p-1}$ & $(\alpha_{p-1}-\beta)^{-1}$ & $\frac{1}{2^q(\alpha_{p-1}-\beta)}$ \\ \hline
    3 & 5 & 7 & 1.90040519768798 & 1.92756197548293 & 36.8232198808926 & 1.15072562127789 \\ 3 & 7 & 11 & 1.92172452309274 & 1.92756197548293 & 171.307607010499 & 1.33834067976952 \\ 3 & 11 & 19 & 1.92717413781454 & 1.92756197548293 & 2578.39833911685 & 1.25898356402190 \\ 3 & 13 & 23 & 1.92745816209718 & 1.92756197548293 & 9632.66916662882 & 1.17586293537949 \\ \hline 5 & 7 & 23 & 1.97926028654319 & 1.98358284342433 & 231.344555433128 & 1.80737933932131 \\ 5 & 13 & 47 & 1.98351307615232 & 1.98358284342433 & 14333.3682296163 & 1.74967873896684 \\ 5 & 19 & 71 & 1.9835169859533 & 1.98358284342433 & 873492.901563983 & 1.66605549156949 \\ \hline 7 & 11 & 59 & 1.99577873757697 & 1.99603117973541 & 3961.30347707098 & 1.93423021341356 \\ 7 & 13 & 71 & 1.99596788607732 & 1.99603117973541 & 15799.3712194387 & 1.92863418206039 \\ 7 & 19 & 107 & 1.99603017934944 & 1.99603117973541 & 999614.177077968 & 1.90661273398964
    \end{tabular}
    \caption{}
    \label{tab:erroranalysis}
\end{table}

\begin{conj}\label{con:limit} The largest limit point of $S$ is $2$. \end{conj}

This would follow from the above work if we could show that, for infinitely many primes $p$, there exists a prime $q>p$ such that $pq-p-q$ is also prime. This would follow, for instance, from Dickson's prime $k$-tuples conjecture, which says that several linear polynomials in $\ZZ[x]$ will be simultaneously prime infinitely often unless there is a congruential obstruction. In this case, for any fixed $p$, we apply this to the two polynomials $x$ and
$(p-1)x-p$, and the conjecture implies there should be infinitely many $x$ where both are prime.
 However, this is stronger than what we need.  Indeed, it suffices to prove  that for infinitely many primes $p$, there
 is at least one value of $x>p$ with both $x$ and $(p-1)x-p$ simultaneously prime.  It may be possible to prove
 this unconditionally. According to our calculations, this appears to be the only route to Conjecture \ref{con:limit}: all points in $S$ close to 2 appear to have this form.

On the other side, there are values of $m$ and $n$ such that $\Phi_m(x)-\Phi_n(x)$ has roots not far from $\pm\frac{1}{2}$. For instance, if $p$ is a large prime, then $\Phi_{3p}(x)-\Phi_4(x)$ has a root near $\rho:=-0.569840290998\cdots$, which is a root of $x^3+x^2+2x+1$.
In fact, as $p\to\infty$, the polynomials $\Phi_{3p}(x)-\Phi_4(x)$ have roots that converge to $\rho$
(and $\Phi_{6p}(x)-\Phi_4(x)$ have roots which converge to $-\rho$).
To see this, note that $\Phi_{3p}(x)=1-x+x^3-x^4+x^6-x^7+\cdots+x^{2p-5}-x^{2p-3}+x^{2p-2}$. As $p\to\infty$, these polynomials converge termwise to the power series $\sum_{n=0}^\infty (1-x)x^{3n}=\frac{1}{1+x+x^2}$. If $|x|<1$, then $\Phi_{3p}(x)\approx\frac{1}{1+x+x^2}$, so $\Phi_{3p}(x)-\Phi_4(x)$ has a root near that of $\frac{1}{1+x+x^2}-\Phi_4(x)$, i.e.,\ where $(1+x+x^2)(1+x^2)=1$. This means that $x^4+x^3+2x^2+x=0$. Curiously, roots of $\Phi_{4p}(x)-\Phi_3(x)$ also converge to the same number.  We can do better however.  The polynomial
$\Phi_{30}(x)-\Phi_4(x)$ has a root at $\sigma:=0.5284555592772\cdots$, and as the prime $p\to\infty$,
$\Phi_{30p}(x)-\Phi_{4p}(x)$ has a root that converges to $\sigma$.  Better still: Take $m$ as the product
of the first $k\ge3$ primes and $n$ as $\frac2{15}m$.  Then $\Phi_m(x)-\Phi_n(x)$ seems to have a root
converging to a number slightly below $0.52$.  For example, when $k=5$, there is a root at
$0.51976982658213\cdots$.
Perhaps the number $\frac12$ in
Theorem \ref{thm:principal} is best possible, but we do not have strong evidence either way.

Based on numerical computations, we present the following conjectures.

\begin{conj} For any distinct positive integers $m$ and $n$, if $z\in\CC\setminus\RR$ and $\Phi_m(z)=\Phi_n(z)$, then $\frac{1}{\sqrt{2}}<|z|\le\sqrt{2}$. The upper bound is attained only for $\{m,n\}=\{1,3\},\{1,4\},\{1,5\}$. \end{conj}

\begin{conj} \label{conj:nonreal} Let $S_{\CC}$ denote the set of all nonreal complex numbers $z$ such that $\Phi_m(z)=\Phi_n(z)$ for some distinct coprime positive integers $m$ and $n$. Then for any $\eps>0$, we have $1-\eps<|z|<1+\eps$ for all but finitely many elements of $S_{\CC}$. \end{conj}

Without the coprime hypothesis Conjecture \ref{conj:nonreal} is likely to be false. 
To see this, note that if $m$ and $n$ are both odd and $\alpha$ is a positive real root of $\Phi_m(x)-\Phi_n(x)$, then $i\alpha^{1/2}$ is a nonreal root of $\Phi_{4m}(x)-\Phi_{4n}(x)$. 
Since presumably polynomials of the form $\Phi_m(x)-\Phi_n(x)$ can have real roots arbitrarily close to 2, this implies that $\Phi_{4m}(x)-\Phi_{4n}(x)$ can have roots arbitrarily close to $\sqrt{-2}$.

However, there are infinitely many real roots bounded away from $\pm 1$.
 Thus we see that apparently there is a significant behavioral difference between the real and nonreal roots of differences of cyclotomic polynomials.

The observed behavior of roots of $\Phi_m(x)-\Phi_n(x)$ is consistent with typical behavior of random polynomials whose coefficients are each chosen uniformly in some large interval. Let $d$ be a large positive integer and $B$ a large positive real number, and let $f(x)$ be a degree-$d$ polynomial in $\RR[x]$ whose coefficients are chosen uniformly and independently from the interval $[-B,B]$. Then it is known (see~\cite{HN08}) that all but $o(d)$ of the roots of $f$ are asymptotically almost surely very close to the unit circle.

On the other hand, the behavior of the real roots, of which there are $o(d)$, behave rather differently. Kac in~\cite{Kac49} showed that the expected number of real roots is $\frac{2}{\pi}\log d$. Similarly, Littlewood and Offord (see~\cite{LO38,LO43,LO45,LO48}) proved that for almost all $f$ (with coefficients chosen independently from any of several different distributions), the number $r_f$ of real roots satisfies \[\frac{\log n}{\log\log\log n}\ll r_f\ll\log^2n.\] Kac also showed that for any $\alpha\in(0,1)$, the expected number of real roots in the range $(0,\alpha)$ is $O(1)$, but not 0.

\section*{Acknowledgments}
We thank Gerry Myerson and Tim Trudgian for bringing Glasby's conjectures to our attention.
We also thank Kevin Ford for reminding us of \cite{BFLPS05}. This project was started at the West Coast Number Theory Conference in Chico, California, in December 2018.

\bibliographystyle{alphaurl}
\bibliography{cyclotomic}

\end{document}